\UseRawInputEncoding
\documentclass[leqno,a4paper,10pt]{amsart}
\usepackage{amsmath,amsthm,amssymb,amsfonts,enumerate,color}
\allowdisplaybreaks[4]

\newtheorem{thm}{Theorem}[section]
\newtheorem{prop}[thm]{Proposition}

\newtheorem{lem}[thm]{Lemma}
\newtheorem{defi}[thm]{Definition}
\newtheorem{remark}[thm]{Remark}
\newtheorem{example}[thm]{Example}
\newtheorem{pb}[thm]{Problem}

\newtheorem{theoA}{Theorem}

\newtheorem{remA}[theoA]{Remark}

\numberwithin{equation}{section}
\newcommand{\R}{{\mathbb R}}

\newcommand{\real}{{\mathbb R}}

\newcommand{\ent}{{\mathbb Z}}

\renewcommand{\L}{{\mathcal L}}

\newcommand{\W}{{\mathcal W}}
\newcommand{\norm}[1]{\left\Vert#1\right\Vert}
\newcommand{\abs}[1]{\left\vert#1\right\vert}

\newcommand{\M}{{\mathcal M}}
\renewcommand{\P}{{\mathcal P}}

\begin{document}

\title[Differential transforms for Poisson semigroups]
{Boundedness of differential transforms for Poisson semigroups generated by Bessel operators}


\author{Chao Zhang }

 \address{School of Statistics and Mathematics \\
             Zhejiang Gongshang University \\
             Hangzhou 310018, P.R. China}
 \email{zaoyangzhangchao@163.com}

\thanks{The  author was supported by the Zhejiang Provincial Natural Science Foundation of China(Grant No. LY22A010011), the National Natural Science Foundation of China(Grant No. 11971431) and the Zhejiang Provincial Social Science Foundation of China(Grant No. 22NDJC094YB)}

\date{}
\maketitle

\begin{abstract}
 In this paper we analyze the convergence of the following type of series
\begin{equation*}  T_N f(x)=\sum_{j=N_1}^{N_2} v_j\Big(\P_{a_{j+1}} f(x)-\P_{a_{j}} f(x)\Big),\quad x\in \mathbb R_+,
\end{equation*}
where  $\{\P_{t} \}_{t>0}$ is the Poisson semigroup of the Bessel operator  $\displaystyle \Delta_\lambda:=-{d^2\over dx^2}-{2\lambda\over x}{d\over dx}$ with $\lambda$ being a positive constant, $N=(N_1, N_2)\in \mathbb Z^2$ with $N_1<N_2,$ $\{v_j\}_{j\in \mathbb Z}$ is a bounded real sequences  and $\{a_j\}_{j\in \mathbb Z}$ is an increasing real sequence.
{Our analysis will consist in the boundedness, in  $L^p(\mathbb{R}_+)$ and in  $BMO(\mathbb{R}_+)$,  of the operators
$T_N$ and its maximal operator $\displaystyle T^*f(x)= \sup_N \abs{T_N f(x)}.$}
It is also shown that the local size of the maximal differential transform operators  is the same with the order  of a singular integral for functions $f$ having local support.
 \end{abstract}


\vskip 1cm

\section{Introduction} \label{Sec:L2}

Let $\lambda$ be a positive constant and $\Delta_\lambda$ be the Bessel operator which is defined by setting, for all suitable functions $f$ on $\mathbb R_+:=(0, \infty)$,
$$\Delta_\lambda f(x):=-{d^2\over dx^2}f(x)-{2\lambda\over x}{d\over dx}f(x).$$
An early work concerning the Bessel operator goes back to Muckenhoupt and Stein \cite{MS}. They developed a theory associated to which is parallel to the classical one associated to the Laplace operator. Since then, a lot of work concerning the Bessel operators was carried out; see, for example \cite{AK, BCFR, BFBMT, BFS, BHNV, DLMWY, DLWY1, DLWY2, MWY, Villani, WYZ, YY}.   In particular, Betancor et al. in \cite{BDT}
 established the characterizations of the atomic Hardy space $H^1((0, \infty), dm_\lambda)$ associated to $\Delta_\lambda$ in terms of the Riesz transform and the radial maximal function related to a class of functions including the Poisson semigroup $\{\P_t\}_{t>0}$ and the heat semigroup $\{\W_t\}_{t>0}$ as special cases, where $dm_\lambda:=x^{2\lambda}dx$ and $dx$ is the Lebesgue measure. The Poisson semigroup operators are defined by
 $$\P_t f(x):=e^{-t\sqrt {\Delta_\lambda}}f(x)=\int_0^\infty\P_t(x,y)f(y)dm_\lambda(y),$$
where $J_v$ is the Bessel function of the first kind of order $v$ with $v\in (-{1\over 2}, \infty)$ and
\begin{align*}\P_t(x,y)&=\int_0^\infty e^{-tz}(xz)^{-\lambda+1/2}J_{\lambda-1/2}(xz)(yz)^{-\lambda+1/2}J_{\lambda-1/2}(yz)dm_\lambda(z)\\
&={2\lambda t\over \pi}\int_0^\pi {(\sin \theta)^{2\lambda-1}\over (x^2+y^2+t^2-2xy\cos \theta)^{\lambda+1}}d\theta, \quad t,x, y\in \mathbb R_+.
\end{align*}

For good enough functions $f$, the families
$\{\P_t f\}_{t>0}, \{\W_t f\}_{t>0} $ converge to $f$  when $t \rightarrow 0^+$. The type of convergence  (pointwise, norm, measure, $\dots$) depends on the class of functions $f$ in consideration.
One way to enter more deeply into this kind  of convergence is to analyze the behavior of the following type sums
\begin{equation}\label{Formu:SquareFun}
 \sum_{j\in \ent} v_j(\P_{a_{j+1}} f(x)-\P_{a_{j}} f(x))  \quad \hbox{and} \quad \sum_{j\in \ent} v_j(\W_{a_{j+1}} f(x)-\W_{a_{j}} f(x)),
\end{equation}
where $\{v_j\}_{j\in \ent}$ is a bounded sequence of real numbers and  $\{a_j\}_{j\in \ent}$ is  an increasing  sequence of positive numbers. Observe that in the case $v_j\equiv1$, the above series is  telescopy, and their behavior coincide with the one of $\P_t f(x)$ and $\W_t f(x)$. This way of analyzing convergence of sequences was considered by Jones and Rosemblatt for ergodic averages(see \cite{JR}), and latter by Bernardis et al. for differential transforms(see \cite{BLMMDT}).

To better understand the behavior of the sums in (\ref{Formu:SquareFun}), we shall analyze its ``partial sums'' defined as follows(we only consider the case for $\P_t$ in this paper).
For each $N\in \ent^2,~N=(N_1,N_2)$ with $N_1<N_2,$ we define
\begin{equation}\label{Formu:FinSquareFun}
 T_N f(x)=\sum_{j=N_1}^{N_2} v_j(\P_{a_{j+1} } f(x)-\P_{a_{j}} f(x)),\ x\in \mathbb R_+.
\end{equation}
 We shall also consider the   maximal operators
\begin{equation}\label{Formu:MaxSquareFun}
 T^*f(x)=\sup_N \abs{T_N f(x)}, \quad x\in\real_+,
\end{equation}
where the supremum are taken over all $N=(N_1,N_2)\in \ent^2$ with $N_1< N_2$.
In \cite{ZMT, ZT}, the authors proved the boundedness of the above operators related with the one-sided fractional Poisson type operator sequence and the classical laplacian separately.

Some of our results will be valid only when the sequence  $\{a_j\}_{j\in \mathbb Z}$ is lacunary.  It means that  there exists a $\rho>1$ such that $\displaystyle \frac{a_{j+1}}{a_j} \ge \rho, \, j \in \mathbb{Z}$.  In particular, we shall prove  the boundedness of the operators $T^*$  in the weighted spaces
$L^p(\mathbb R_+, \omega dm_\lambda)$ and $BMO(\mathbb R_+, dm_\lambda)$ space, where $\omega$ is  the usual Muckenhoupt weight on $\mathbb R_+$. We refer the reader to the book by J. Duoandikoetxea \cite[Chapter 7]{Duo} for definitions and properties of the $A_p$ classes. We say that, a function   $f\in L_{loc}^1(\real_+, dm_\lambda) $  belongs to the space $BMO(\real_+, dm_\lambda)$ if
$$\norm{f}_{BMO(\R_+, dm_\lambda)}  := \sup_{x,r \in (0, \infty)} \Big\{ \frac1{m_\lambda(I(x,r))} \int_{I(x,r)}\Big|f(y)- f_{I(x,r)}\Big|\  dm_\lambda(y)  \Big\}<\infty,$$
where $$f_{I(x,r)}:={1\over m_\lambda(I(x,r))}\int_{I(x,r)}f(y)\  dm_\lambda(y). $$
And we have the following results.
\begin{thm}\label{Thm:PoissonLp}    Assume that the sequence $\{a_j\}_{j\in \mathbb Z}$ is a $\rho$-lacunary sequence with $\rho >1$. Let $T^*$ be   the operator   defined in \eqref{Formu:MaxSquareFun}. We have the following statements.
\begin{enumerate}[(a)]
    \item For any $1<p<\infty$ and $\omega\in A_p$,  there exists a constant $C$ depending  on $\rho, p,\omega,  \lambda$ and $\norm{v}_{l^\infty(\mathbb Z)}$ such that
 $$\norm{T^* f}_{L^p(\mathbb R_+, \omega dm_\lambda)}\leq C\norm{f}_{L^p(\mathbb R^{n}, \omega dm_\lambda)},$$
 for all functions $f\in L^p(\real_+, \omega dm_\lambda).$
    \item For any  $\omega\in A_1$, there exists a constant $C$ depending  on $\rho, \omega, \lambda$ and  $\norm{v}_{\ell^\infty(\mathbb Z)}$   such that
 $$\omega\left({\{x\in \real_+:\abs{T^* f(x)}>\sigma\}}\right) \le C\frac{1}{\sigma}\norm{f}_{L^1(\mathbb R_+, \omega dm_\lambda)}, \quad \sigma>0,$$
for all functions $f\in L^1(\real_+, \omega  dm_\lambda).$

\item Let $\{a_j\}_{j\in \mathbb Z}$ be  a $\rho$-lacunary sequence. Given $f\in L^\infty(\real_+, dm_\lambda),$ then either $T^* f(x) =\infty$ for all $x\in \mathbb R_+$, or $T^* f(x) < \infty$ for $a. e.$  $x\in \mathbb R_+$. And in this latter case, there exists a constant $C$ depending  on $\lambda, \rho$ and  $\norm{v}_{\ell^\infty(\mathbb Z)}$ such that
$$\norm{T^*f}_{BMO(\mathbb R_+, dm_\lambda)}\leq C\norm{f}_{L^\infty(\mathbb R_+, dm_\lambda)}.$$
\item Let $\{a_j\}_{j\in \mathbb Z}$ be a $\rho$-lacunary sequence. Given $f\in BMO(\real_+),$ then either $T^* f(x) =\infty$ for all $x\in \mathbb R_+$, or $T^* f(x) < \infty$ for $a. e.$  $x\in \mathbb R_+$.  And in this latter case,  there exists a constant $C$ depending  on $\rho, \lambda$ and  $\norm{v}_{\ell^\infty(\mathbb Z)}$ such that
\begin{equation}\label{sharp}\norm{T^*f}_{ BMO(\mathbb R_+, dm_\lambda)}\leq C\norm{f}_{BMO(\mathbb R_+)}.
\end{equation}
\end{enumerate}
\end{thm}

\begin{remark}
  From the conclusions  in Theorem \ref{Thm:PoissonLp}, for $f\in L^p(\real_+, \omega dm_\lambda)$ with $\omega\in A_p$,  we can define $T f$ by the limit of $T_N f$ in $L^p(\real_+, \omega dm_\lambda)$ norm
  $$T f(x)=\lim_{(N_1,N_2)\rightarrow (-\infty, +\infty)} T_N f(x),\quad \quad ~x\in \real_+.$$
 For more results related with the convergence of  $T_N f$, see Proposition \ref{Thm:ae}.
\end{remark}

In classical Harmonic Analysis,  if $f= \chi_{(0,1)}$ and $\mathcal{H}$ is the Hilbert transform, it is known  that   $\displaystyle  \frac1{r} \int_{-r}^0 \mathcal{H}(f)(x)dx\  \sim \log\frac{e}{r}$  as $ r \to 0^+$. In general, this is the growth of a singular integral applied to a bounded function  at the origin.
Some related results  about the local behavior of variation operators can be found in \cite{BCT}. One dimensional results about the variation of  convolution operators can be found in \cite{MTX}. And  one dimensional results about  differential transforms of one-sided fractional Poisson type operator sequence is proved in \cite{ZMT}.  The related results about  differential transforms of heat semigroup generated by classical Laplaciane is proved in \cite{ZT}.

The following theorem analyzes the behavior of $T^*$ in $L^\infty$.

\begin{thm} \label{Thm:GrothLinfinity}
Let    $\{v_j\}_{j\in \mathbb Z}\in l^p(\mathbb Z)$ for some $1 \le  p\le \infty$, $\{a_j\}_{j\in \mathbb Z}$ be any increasing sequence and  $T^*$  be defined in (\ref{Formu:MaxSquareFun}). Then for every $f\in L^\infty(\mathbb{R}_+, dm_\lambda)$ with support in the  interval $\tilde I=(0, 1)$,  for any interval $\tilde I_r:=(0, r)$ with $2r<1$, we have
    $$\frac{1}{m_\lambda(\tilde I_r)} \int_{\tilde I_r} \abs{T^* f (x)} dm_\lambda(x)\leq C\left(\log \frac{2}{r}\right)^{1/p'}\norm{v}_{l^p(\mathbb Z)}\|f\|_{L^\infty(\mathbb R_+, dm_\lambda)}.$$
In the statement above,  $\displaystyle p' = \frac{p}{p-1},$ and if $p=1$, $\displaystyle p'=\infty.$
\end{thm}

This article is organized as follows. In Section \ref{Sec:Laplacian}, we will get the kernel estimates to see that the kernel $K_N$ of $T_N$ is a vector-valued Caldr\'on-Zygmund kernel, and then we can get the uniform boundedness of $T_N.$ And with a Cotlar's inequality, we can get the proof of Theorem   \ref{Thm:PoissonLp} in Section \ref{Sec:boundednessT}. And we prove Theorem \ref{Thm:GrothLinfinity} in the last section.

Throughout this article, the letters $C, c$ will denote  positive
constants which may change from one instance to another and depend on
the parameters involved. We will make a frequent use, without
mentioning it in relevant places, of the fact that for a positive
$A$ and a non-negative $a,$
$$\sup\limits_{t>0}t^a
e^{-At}=C_{a, A}<\infty.$$ For any $k\in \mathbb R_+$ and $I:=I(x,r)$ for some $x,r\in (0, \infty),$ $kI:=I(x, kr),$ if $x<r,$ then $$I(x,r)=(0, x+r)=I({x+r\over 2}, {x+r\over 2}).$$ Thus, we may assume that $x\ge r.$




\vskip 1cm


\section{Uniform boundedness of the operatrs $T_N$}\label{Sec:Laplacian}

In this section, we  give the proofs of Theorem \ref{Thm:PoissonLp}. {Vector-valued  Calder\'on-Zygmund theory will be a fundamental tool  in proving the $L^p$ boundedness of the operators  $T_N$. For this theory, the reader can see  the paper \cite{RubioRuTo}. Nowadays it is well known that the fundamental ingredients in the theory are the $L^{p_0}(\mathbb{R}^n)$ boundedness for some $1<p_0<\infty$ and the smoothness of the kernel of the operator. Even more, the constants that appear in the results only depend on the boundedness constant in $L^{p_0}(\mathbb{R}^n)$ and the constants related with the size and smoothness of the kernel.}

 Let  $\lambda\in (0, \infty).$ For $y\in \R_+,$ consider the functions
 $$\varphi_y(x)=(xy)^{-\lambda+1/2}J_{\lambda-1/2}(xy),\quad x\in \R_+,$$
where $J_v$ denotes the Bessel function of the first kind and order $v$,  see \cite{Watson}.  It is well known that for each $y\in \R_+$, the function $\varphi_y$ is an eigenfunction of the Bessel operator $\Delta_\lambda,$ and the corresponding eigenvalue is $|y|^2$,
$$\Delta_\lambda \varphi_y=|y|^2\varphi_y, \quad y\in \R_+.$$
  The set $\{\varphi_y(x)\}_{y\in (0, \infty)}$ of eigenfunctions of $\Delta_\lambda$, does not span a dense subset of $L^2(\mathbb R_+, dm_\lambda).$ Thus, we cannot use the usual spectral techniques to define the classical operators associated with $\Delta_\lambda.$  The Hankel transform $H_\lambda$ defined by $$H_\lambda f(y)=\int_{\R_+} \varphi_y(x)f(x)dm_\lambda(x),\quad y\in \R_+,$$
  plays in the Bessel context a similar role as the Fourier transform in the Euclidean setting. It is well known that $H_\lambda$ is an isometry in $L^2(\R_+,dm_\lambda)$ and it coincides there with its inverse, $H_\lambda^{-1}=H_\lambda.$ Moreover, for sufficiently regular functions $f$, say $f\in C_c^\infty(\R_+),$ we have
  $$H_\lambda(\Delta_\lambda f)(y)=|y|^2H_\lambda(f)(y),\quad y\in \R_+.$$

  We consider the nonnegative self-adjoint extension of $\Delta_\lambda$(still denoted by the same symbol) defined by $$\Delta_\lambda f=H_\lambda(|y|^2H_\lambda f), \quad f\in \text{Dom}(\Delta_\lambda),$$
  on the domain $$\text{Dom}(\Delta_\lambda)=\{f\in L^2(\R_+, dm_\lambda): |y|^2H_\lambda f\in L^2(\R_+, dm_\lambda)\}.$$
 Then the spectral decomposition of $\Delta_\lambda$ is given via the Hankel transform.

  In the following theorem, we present and prove  the $L^2$ boundedness of the operators $T_N$ by Hankel transform.

\begin{prop}\label{Thm:L2Estimate}
There is a constant $C$, depending  on $\lambda, \norm{v}_{\ell^\infty(\mathbb Z)}$(not on $N$), such that
 $$ \|T_N f \|_{L^2(\real_+, dm_\lambda)}\leq C \|f \|_{L^2(\real_+, dm_\lambda)}.$$
\end{prop}

\begin{proof}
Let $f\in \text{Dom}(\Delta_\lambda)$. Using  the Plancherel theorem for Hankel transform, we have
\begin{align*}
 \norm{T_N f }^2_{L^2(\real_+, dm_\lambda)} & = \norm{\sum_{j=N_1}^{N_2} v_j\left(\P_{a_{j+1}} f -\P_{a_{j}}f\right)}^2_{L^2(\real_+, dm_\lambda)}\\
 &=  \int_{\mathbb{R}_+} \Big\{\sum_{j=N_1}^{N_2} v_j\left(\P_{a_{j+1}} f(x) -\P_{a_{j}}f(x)\right) \Big\}^2 dm_\lambda(x) \\
  &=
 \int_{\mathbb{R}_+} \left(H_\lambda{\Big\{\sum_{j=N_1}^{N_2} v_j\left(\P_{a_{j+1}} f(\cdot) -\P_{a_{j}}f(\cdot)\right) \Big\}}(y)\right)^2 dm_\lambda(y)\\
  & =  \int_{\mathbb{R}_+} \Big\{\sum_{j=N_1}^{N_2} v_j \int_{a_j}^{a_{j+1}}\partial_t H_\lambda \left(\P_t f \right)(y)dt  \Big\}^2 dm_\lambda(y)\\
 &\le
   C\norm{v}^2_{\ell^\infty(\mathbb Z)}
   \int_{\mathbb{R}_+} \Big\{\sum_{j=N_1}^{N_2} \Big|\int_{a_j}^{a_{j+1}}\partial_t H_\lambda \left(\P_t f\right)(y) dt\Big|  \Big\}^2 dm_\lambda(y)\\  &=
   C_v
   \int_{\mathbb{R}_+} \Big\{\sum_{j=N_1}^{N_2} \Big|\int_{a_j}^{a_{j+1}}\abs{y} e^{-t\abs{y}} H_\lambda ( f)(y)dt\Big|  \Big\}^2 dm_\lambda(y)\\
   &\le    C_v
   \int_{\mathbb{R}_+} \Big\{\sum_{j=N_1}^{N_2} \int_{a_j}^{a_{j+1}}\abs{y} e^{-t\abs{y}} dt|H_\lambda( f)(y)|  \Big\}^2 dm_\lambda(y)\\ &\le  C_v
   \int_{\mathbb{R}_+} \Big\{\Big|\int_0^\infty\abs{y} e^{-t\abs{y}} dt\Big||H_\lambda( f)(y)|  \Big\}^2 dm_\lambda(y)\\
   & \le  C_{v, \lambda} \|f\|_{L^2(\real_+, dm_\lambda)}^2.
\end{align*}
Since $\text{Dom}(\Delta_\lambda)$ is dense in $L^2(\real_+, dm_\lambda),$ $T_N$ extends uniquely to  a bounded linear operator on $L^2(\real_+, dm_\lambda).$
Then the proof of the theorem is complete.
\end{proof}

It is straightforward from the definition of $m_\lambda$ that there exists a finite constant $C>1$ such that for all $x,r \in \mathbb R_+,$
$$C^{-1} m_\lambda(I(x,r))\le x^{2\lambda}r+r^{2\lambda+1}\le Cm_\lambda(I(x,r)).$$
This means that $(\mathbb R_+, |\cdot|, dm_\lambda)$ is a space of homogeneous type in the sense of \cite{CW1, CW2}.
In order to use the Calder\'on-Zygmund theory, we need prove a proposition  containing the description of the kernel. To begin with, we need the following lemma on the upper bounds of the Poisson kernel and its derivatives, which is useful tool in the later proofs.
\begin{lem}[See {\cite[Proposition 2.1]{WYZ}}]\label{KerPoisson}
There exists a positive constant $C$ such that for any $x,y,t\in (0, \infty)$,
\begin{enumerate}[\indent i)]
\item \begin{equation*}
\abs{\P_t(x,y)}\le C{t\over {(|x-y|^2+t^2)^{\lambda+1}}},
\end{equation*}\\
and \begin{equation*}
\abs{\P_t(x,y)}\le C{t\over {(xy)^{\lambda}(|x-y|^2+t^2)}}.
\end{equation*}\\
\item   \begin{equation*}
\abs{\partial_x\P_t(x,y)}\le C{t\over {(|x-y|^2+t^2)^{\lambda+{3\over 2}}}},
\end{equation*}
and \begin{equation*}
\abs{\partial_x\P_t(x,y)}\le C{t\over {(xy)^{\lambda}(|x-y|^2+t^2)^{3\over 2}}}.
\end{equation*}
\item   \begin{equation*}
\abs{\partial_t\P_t(x,y)}\le C{1\over {(|x-y|^2+t^2)^{\lambda+{1}}}},
\end{equation*}
and \begin{equation*}
\abs{\partial_t\P_t(x,y)}\le C{1\over {(xy)^{\lambda}(|x-y|^2+t^2)}}.
\end{equation*}
\item   \begin{equation*}
\abs{\partial_y\partial_t\P_t(x,y)}+ \abs{\partial_x\partial_t\P_t(x,y)}\le C{1\over {(|x-y|^2+t^2)^{\lambda+{3\over 2}}}},
\end{equation*}
and \begin{equation*}
\abs{\partial_y\partial_t\P_t(x,y)}+\abs{\partial_x\partial_t\P_t(x,y)}\le C{1\over {(xy)^{\lambda}(|x-y|^2+t^2)^{3\over 2}}}.
\end{equation*}
\end{enumerate}

\end{lem}

\begin{prop}\label{Thm:KernelEst}
Let $f\in L^p(\mathbb{R}_+, dm_\lambda), 1\le p \le \infty$. Then
$$T_Nf(x) = \int_{\mathbb{R}_+} K_N(x,y)f(y)dm_\lambda(y)$$
with
\begin{align*}\label{kernel}
&K_N(x, y) =  \sum_{j=N_1}^{N_2}v_j \left(\P_{a_{j+1}} (x,y)-\P_{a_{j}}(x,y)\right)\\
&={2\lambda\over \pi}\sum_{j=N_1}^{N_2}v_j \left(\int_0^\pi{a_{j+1}(\sin \theta)^{2\lambda-1}\over {(x^2+y^2+a_{j+1}^2-2xy\cos \theta)^{\lambda+1}}}d\theta -\int_0^\pi{a_{j}(\sin \theta)^{2\lambda-1}\over {(x^2+y^2+a_{j}^2-2xy\cos \theta)^{\lambda+1}}}d\theta\right).
\end{align*}
Moreover, there exists  constant $C>0$ depending  on $\lambda$  and $\norm{v}_{\ell^\infty(\mathbb Z)}$(not on $N$) such that, for any $x\neq y,$
\begin{enumerate}[\indent i)]
  \item $\displaystyle   |K_N(x,y)|\leq \frac{C}{m_\lambda(I(x,|x-y|))}$,
  \item $\displaystyle   |\partial_x K_N(x,y)|+ |\partial_y K_N(x,y)|\leq \frac{C}{m_\lambda(I(x,|x-y|))\cdot |x-y|}$.
\end{enumerate}
\end{prop}

\begin{proof}
{\it i)}~Regarding  the size condition for the kernel, we consider it into two cases:

 Case 1: $x\le 2|x-y|$. In this case, $$m_\lambda{(I(x, |x-y|))}\sim |x-y|^{2\lambda+1}.$$    By Lemma \ref{KerPoisson} $iii)$,  we have
\begin{align*}
  |K_N (x, y)| &\leq  \sum_{j=N_1}^{N_2}\abs{v_j} \abs{\P_{a_{j+1}} (x,y)-\P_{a_{j}(x,y)}(x,y)}\\
&\le C_{\lambda, v} \sum_{j=-\infty}^{\infty}  \left|\int_{a_j}^{a_{j+1}}\partial_t  \P_t(x, y) dt\right|\\
   &\le C_{\lambda, v}\int_{0}^{\infty}{1\over {(|x-y|^2+t^2)^{\lambda+{1}}}}{dt } \\
    &\le  C_{\lambda, v}{1 \over  {|x-y|^{2\lambda+1}}}=\frac{C}{m_\lambda(I(x,|x-y|))}.
\end{align*}

Case 2: $x>2|x-y|$. In this case, $\displaystyle {x\over 2}\le y\le {3x\over 2}$ and $$m_\lambda{(I(x, |x-y|))}\sim  x^{2\lambda}|x-y|\sim  (xy)^{\lambda}|x-y|.$$    By Lemma \ref{KerPoisson} $iii)$,  we have
\begin{align*}
  |K_N (x, y)| &\leq \frac{2\lambda}{ \pi} \sum_{j=N_1}^{N_2}\abs{v_j} \abs{\P_{a_{j+1}} (x,y)-\P_{a_{j}(x,y)}(x,y)}\\
&\le C_{\lambda, v} \sum_{j=-\infty}^{\infty}  \left|\int_{a_j}^{a_{j+1}}\partial_t  \P_t(x, y) dt\right|\\
   &\le C_{\lambda, v}\int_{0}^{\infty}{1\over {(xy)^\lambda(|x-y|^2+t^2)}}{dt } \\
    &\le  C_{\lambda, v}{1 \over  {(xy)^\lambda |x-y|}}=\frac{C}{m_\lambda(I(x,|x-y|))}.
\end{align*}

{\it ii)}~
 With a similar argument as above in $i)$,  we can get the proof of $ii)$ but by using the estimates in Lemma \ref{KerPoisson} $iv).$
The proof of the  proposition is complete.
\end{proof}

\begin{thm}\label{Thm:BMO}  Let $T_N$ be  the operator defined in (\ref{Formu:FinSquareFun}). We have the following statements.
\begin{enumerate}[(a)]
\item For any $1<p<\infty$ and $\omega\in A_p$,  there exists a constant $C$ depending  on $ p,\omega,  \lambda$ and $\norm{v}_{l^\infty(\mathbb Z)}$ such that
 $$\norm{T_N f}_{L^p(\mathbb R_+, \omega dm_\lambda)}\leq C\norm{f}_{L^p(\mathbb R¡ª¡ª+, \omega dm_\lambda)},$$
 for all functions $f\in L^p(\real_+, \omega dm_\lambda).$
    \item  For any  $\omega\in A_1$, there exists a constant $C$ depending  on $ \omega, \lambda$ and $\norm{v}_{\ell^\infty(\mathbb Z)}$    such that
 $$\omega\left({\{x\in \real_+:\abs{T_N f(x)}>\sigma\}}\right) \le C\frac{1}{\sigma}\norm{f}_{L^1(\mathbb R_+, \omega dm_\lambda)}, \quad \sigma>0,$$
for all functions $f\in L^1(\real_+, \omega  dm_\lambda).$
    \item Let $\{a_j\}_{j\in \mathbb Z}$ be a positive increasing sequence. There exists a constant $C$ depending on $\lambda$ and $\norm{v}_{\ell^\infty(\mathbb Z)}$ such that
$$\norm{T_N f}_{BMO(\mathbb R_+, dm_\lambda)}\leq C\norm{f}_{L^\infty(\mathbb R_+, dm_\lambda)},$$
for all functions $f\in L^\infty(\real_+).$
\item Let $\{a_j\}_{j\in \mathbb Z}$ be a positive increasing sequence. There exists a constant $C$ depending  on $\lambda$ and $\norm{v}_{\ell^\infty(\mathbb Z)}$ such that
     $$\norm{T_N f}_{BMO(\real_+, dm_\lambda)}\le C\norm{f}_{BMO(\real_+, dm_\lambda)}.$$

\end{enumerate}

The constants $C$ appeared above all are independent of $N.$
\end{thm}
\begin{proof}
By Propositions \ref{Thm:L2Estimate} and \ref{Thm:KernelEst}, we know that the operator $T_N$ is a Calder\'on-Zygmund operator with standard kernel. Previously, we have remarked  that the constants  in the $L^p$ boundedness only depend on the initial constant in $L^{p_0}(\mathbb{R}_+)$(in our case $p_0=2$), the size constant and smoothness constant of the kernel.  Hence the uniform boundedness of the operators $T_N$ in $L^p(\mathbb{R}_+, dm_\lambda)$ spaces is a direct consequence of the vector-valued Calder\'on-Zygmund theory in homogeneous type space $(\mathbb R_+, |\cdot|, dm_\lambda)$.  So, it is bounded on the weighted $L_p$-spaces for $1<p<\infty$, and is of weak type $(1,1)$. Then, we get the proof of $(a)$ and $(b)$.

For $(c)$ and $(d),$ we can prove them as following.
 The finiteness of $T_N$ for functions in  $L^\infty(\real_+, dm_\lambda)$ is obvious, since for each $N$, $K_N$ is an integrable function.
  On the other hand, if $f\in BMO(\real_+, dm_\lambda)$, we can proceed as follows. Let $I=I(x_0, r_0)$ and $I^*=2I=I(x_0, 2r_0)$ with some $x_0\in \real_+$ and $r_0>0$. We decompose $f$ to  be
$$f=(f-f_{I })\chi_{I^*}+(f-f_{I })\chi_{(I^*)^c}+f_{I }=:f_1+f_2+f_3,$$
where $\displaystyle f_{I }={1\over m_\lambda(I(x_0,r_0))}\int_{I(x_0,r_0)}f(y)\  dm_\lambda(y).$
The function $f_1$ is integrable, hence  $T_N f_1(x)<\infty,$ $a.e.\ x\in \real_+.$ For $T_N f_2$, we note that, for any $x\in I$,   $y\in \{y\in \mathbb R_+: 2^k r_0<|x_0-y|\le 2^{k+1}r_0\}$ with  $k\ge 1$, $|x-y|\sim |y-x_0|$ and we claim that
\begin{equation}\label{eq:max}
 \max \left\{ {t\over{(|x-y|^2+t^2)^{\lambda+1}}},\ {t\over {(xy)^{\lambda}(|x-y|^2+t^2)}}    \right\}\le C {t\over {m_\lambda(I(x_0, 2^kr_0))\cdot 2^kr_0}}.
\end{equation}
In fact, we can prove \eqref{eq:max} by considering the following two cases:

Case 1: $x_0\le 4\cdot 2^{k}r_0.$ In this case, $$m_\lambda(I(x_0, 2^kr_0))\sim (2^k r_0)^{2\lambda+1}.$$
So
$${t\over{(|x-y|^2+t^2)^{\lambda+1}}}\le {t\over{|x-y|^{2\lambda+2}}}\le C {t\over{(2^kr_0)^{2\lambda+2}}}=C {t\over {m_\lambda(I(x_0, 2^kr_0))\cdot 2^kr_0}}.$$

Case 2: $x_0 >4\cdot  2^{k}r_0.$ In this case,  $$m_\lambda(I(x_0, 2^kr_0))\sim x_0^{2\lambda+1}2^k r_0.$$
Since $x\in I$ and $y\in \{y\in \mathbb R_+: 2^k r_0<|x_0-y|\le 2^{k+1}r_0\}$, $x\sim y\sim x_0.$ And then, we get
$${t\over {(xy)^{\lambda}(|x-y|^2+t^2)}}\le {t\over{x_0^{2\lambda}|x_0-y|^{2}}}\le C {t\over{(x_0^{2\lambda}2^kr_0)\cdot 2^kr_0}}=C {t\over {m_\lambda(I(x_0, 2^kr_0))\cdot 2^kr_0}}.$$

Then, for any $t>0,$ by \eqref{eq:max} we have
\begin{align*}
\abs{\P_t f_2(x)}&= \abs{\int_{\real_+}  \P_t(x,y) f_2 (y)dm_\lambda(y)}\\
 &\le C \sum_{k=1}^\infty\Big|\int_{\{2^kr_0<|x_0-y|\le 2^{k+1}r_0\}\cap \mathbb R_+}
           \max \left\{ {t\over{(|x-y|^2+t^2)^{\lambda+1}}},\ {t\over {(xy)^{\lambda}(|x-y|^2+t^2)}}    \right\}\\
           &\quad\quad\quad\quad\quad\quad\quad\quad\quad\quad\quad\quad\quad\quad
           \quad\quad\quad\quad\quad\quad\quad\quad\quad\quad\quad\quad\quad\quad\quad\quad \cdot ({f(y)-f_{I}})dm_\lambda(y)\Big|\\
&\le Ct\sum_{k=1}^\infty{(2^kr_0)}^{-1}{1\over {m_\lambda(I(x_0, 2^kr_0))}}\abs{\int_{\{|x_0-y|\le 2^{k+1}r_0\}\cap \mathbb R_+}  ({f(y)-f_{I}})dm_\lambda(y)}\\
&\le  Ct  \sum_{k=1}^\infty{(2^kr_0)}^{-1}(1+ k)\norm{f}_{BMO(\real_+, dm_\lambda)}<\infty.
\end{align*}
So, $\P_t f_2(x)$ is finite for any $x\in I$ and $t>0.$ Since $T_N f_2(x)$ is a finite summation and $x_0, r_0$ is arbitrary, $T_N f_2(x)<\infty$ $a.e.$ $x\in \real_+.$
Finally we note that $T_N f_3(x)\equiv 0$, since $\P_{a_j} f_3= f_{I}$ for any $j\in \mathbb Z.$
Hence, $T_N f(x)<\infty$ $a.e.$ $x\in \real_+.$ Then, by Propositions \ref{Thm:L2Estimate} and  \ref{Thm:KernelEst}  we  get the proof of part $(c)$ of   Theorem \ref{Thm:BMO}. Since $T_N1=0$, we can get the  conclusion of $(d)$.
\end{proof}

\section{Boundedness of the Maximal  operator $T^*$}\label{Sec:boundednessT}
In this section, we will give the proof of Theorem \ref{Thm:PoissonLp} related to the boundedness of the maximal operator $T^*$. The next lemma,  parallel to  Proposition 3.2 in \cite{BLMMDT}(also Proposition 3.1 in \cite{ZMT}), shows that, without lost of generality, we may assume that
\begin{equation}\label{equ:lacunary}
1<\rho \leq {a_{j+1} \over a_j}\leq \rho^2, \quad j\in \mathbb Z.
\end{equation}

\begin{lem}\label{Prop:lacunary}
Given a $\rho$-lacunary sequence $\{a_j\}_{j\in \mathbb Z}$ and a multiplying sequence $\{v_j\}_{j\in \mathbb Z}\in \ell^\infty(\mathbb Z)$, we can define a $\rho$-lacunary sequence $\{\eta_j\}_{j\in \mathbb Z}$ and $\{\omega_j\}_{j\in \mathbb Z}\in \ell^\infty(\mathbb Z)$ verifying the following properties:
\begin{enumerate}[(i)]
\item $1<\rho \leq \eta_{j+1}/\eta_j\leq \rho^2,\quad \norm{\omega_j}_{\ell^\infty(\mathbb Z)}=\norm{v_j}_{\ell^\infty(\mathbb Z)}$.
\item For all $N=(N_1, N_2)$, there exists $N'=(N_1', N_2')$ with $T_N=\tilde{T}_{N'},$
where $\tilde{T}_{N'}$ is the operator defined in \eqref{Formu:FinSquareFun} for the new sequences $\{\eta_j\}_{j\in \ent}$ and $\{\omega_j\}_{j\in \ent}.$
\end{enumerate}
\end{lem}

\begin{proof} We  follow closely the ideas in the proof
of Proposition 3.2 in \cite{BLMMDT}. We include them here for completeness.

Let $\eta_0=a_0$, and let us construct $\eta_j$ for positive $j$ as follows(the argument for
negative $j$ is analogous). If $\rho^2\ge {a_1/ a_0}\ge \rho,$ define $\eta_1=a_1$. In the opposite case where
$a_1/a_0>\rho^2,$  let $\eta_1=\rho a_0.$ It verifies $\rho^2\ge \eta_1/\eta_0=\rho\ge \rho.$ Further, $a_1/\eta_1\ge \rho^2 a_0/\rho a_0=\rho.$  Again, if $a_1/\eta_1\le \rho^2,$ then $\eta_2=a_1.$ If this is not the case, define $\eta_2=\rho^2 a_0\le a_1$.  By
the same calculations as before, $\eta_0, \eta_1, \eta_2$ are part of a lacunary sequence satisfying \eqref{equ:lacunary}.
To continue the sequence, either $\eta_3=a_1$ (if $a_1/ \eta_2\le \rho^2$) or $\eta_2=\rho^3\eta_0$ (if $a_1/\eta_2>\rho^2$). Since $\rho>1,$ this process ends at some $j_0$ such that $\eta_{j_0}=a_1.$ The rest of the elements $\eta_j$
are built in the same way, as the original $a_k$ plus the necessary terms put in between two
consecutive $a_j$ to get \eqref{equ:lacunary}.
Let $J(j)=\{k:a_{j-1}<\eta_j\le a_j\}$, and $\omega_k=v_j$ if $k\in J(j)$. Then
\begin{equation*}
 v_j(\P_{a_{j+1}} f(x)-\P_{a_{j} } f(x))=\sum_{k\in J(j) }\omega_k(\P_{a_{k+1} } f(x)-\P_{a_{k} } f(x)).
\end{equation*}
If $M=(M_1, M_2)$ is the number such that $\eta_{M_2}=a_{N_2}$ and $\eta_{M_1-1}=a_{N_1-1}$, then we get
\begin{equation*}
 T_N f(x)=\sum_{j=N_1}^{N_2} v_j(\P_{a_{j+1}  } f(x)-\P_{a_{j} } f(x))=\sum_{k=M_1}^{M_2} \omega_k(\P_{a_{k+1}  } f(x)-\P_{a_{k} } f(x))= \tilde{T}_M f(x),
\end{equation*}
where $\tilde{T}_M$  is the operator defined in \eqref{Formu:FinSquareFun} related with sequences $\{\eta_k\}_{k\in \mathbb Z}$, $\{\omega_k\}_{k\in \mathbb Z}$ and $M=(M_1, M_2)$.
\end{proof}

This proposition allows us to assume in the rest of the article that the  lacunary sequences
 $\{a_j\}_{j\in \mathbb Z}$  satisfy \eqref{equ:lacunary} without saying it explicitly.

In order to prove Theorem \ref{Thm:PoissonLp} for the case of the classical laplacian we shall need a Cotlar's type inequality to control the operator $T^*$.
Namely, we shall prove the following theorem.
\begin{thm}\label{Thm:Maximalcontrol}
For each $q\in (1, +\infty),$ there exists a constant $C$ depending  on $\lambda, \norm{v}_{\ell^\infty(\mathbb Z)}$and $\rho$  such that,  for every $x\in \real_+$ and every $M\in \mathbb Z^+$,
\begin{equation*}
T_M^*f(x)\le C\left\{\M(T_{(-M, M)} f)(x)+\M_q f(x)\right\},
\end{equation*}
where $$T_M^{*}f(x)=\sup_{-M\le N_1<N_2\le M}\abs{ T_N f(x)}$$ and
\begin{equation}\label{Mq}
\M_qf(x)=\sup_{r>0} \left(\frac{1}{m_\lambda(I(x, r))}\int_{I(x, r)}\abs{f(y)}^qdy\right)^{1\over q},\quad 1< q<\infty.\end{equation}
\end{thm}

In order to prove the above theorem, we shall need the following lemma.

\begin{lem}\label{cotlar}  Let  $\{a_j\}_{j\in \mathbb Z}$  a $\rho$-lacunary sequence and $\{v_j\}_{j\in \mathbb Z} \in \ell^\infty(\mathbb Z)$. Then
\begin{itemize}
\item[(i)] $\displaystyle \abs{\sum_{j=m}^{M}v_j \left(\P_{a_{j+1}}(x,y)-\P_{a_{j}}(x,y) \right) } \le { {C_{v, \rho}} \over m_\lambda(I(x,a_m))}, \quad \forall x\in \mathbb R_+,  \ |x-y|\le a_m,$
\

\item[(ii)] if $k\ge m$ and $x,y \in\mathbb{R}_+$ with $a_k\le |x-y|\le  a_{k+1}$, then \begin{align*}
\Big|\sum_{j=-M}^{m-1}v_j \left(\P_{a_{j+1}}(x,y)-\P_{a_{j}}(x,y)  \right) 	 \Big|\, \le C_{v, \rho}\frac1{m_\lambda(I(x,a_k))}\rho^{-(k-m+1)}.
\end{align*}

\end{itemize}

\end{lem}
\begin{proof}(i). We consider the following two cases:

Case 1: $x\le 2a_m.$ In this case,
$$m_\lambda(I(x,a_m))\sim a_m^{2\lambda+1}.$$
By the mean value theorem and Lemma \ref{KerPoisson} $iii)$,  there exists  $a_j\le \xi_j\le a_{j+1}$ such that
\begin{align*}
&\abs{\sum_{j=m}^{M}v_j \left(\P_{a_{j+1}}(x,y)- \P_{a_{j}}(x,y) \right) }\le
  C\norm{v}_{l^\infty(\mathbb Z)} \sum_{j=m}^{M} (a_{j+1}-a_j)\abs{\partial_t \P_{t}(x,y)\big|_{t=\xi_j}}\\  &\le C_{v}  \sum_{j=m}^{M} (a_{j+1}-a_j) \abs{1\over (|x-y|^2+\xi_j^2)^{\lambda+1}} \le C_{v}  \sum_{j=m}^{M} (\rho^2-1) {1\over \abs{a_j
   }^{2\lambda+1}} \\ &  \le  C_{v, \rho} \frac1{a_m^{2\lambda+1}} \sum_{j=m}^{M}  \frac1{\rho^{(2{\lambda}+1)(j-m)} }  \le C_{v, \rho} {1\over m_\lambda(I(x,a_m))},
\end{align*}
where we have used $\displaystyle \rho\le {a_{j+1}\over a_j}\le \rho^2.$

Case 2: $x\ge 2a_m.$ In this case,
$$m_\lambda(I(x,a_m))\sim x^{2\lambda}a_m\quad \text{and} \quad x\sim y.$$
Also, by the mean value theorem and Lemma \ref{KerPoisson} $iii)$,  there exists  $a_j\le \xi_j\le a_{j+1}$ such that
\begin{align*}
&\abs{\sum_{j=m}^{M}v_j \left(\P_{a_{j+1}}(x,y)- \P_{a_{j}}(x,y) \right) }\le
  C\norm{v}_{l^\infty(\mathbb Z)} \sum_{j=m}^{M} (a_{j+1}-a_j)\abs{\partial_t \P_{t}(x,y)\big|_{t=\xi_j}}\\  &\le C_{v}  \sum_{j=m}^{M} (a_{j+1}-a_j) \abs{1\over (xy)^\lambda(|x-y|^2+\xi_j^2)} \le C_{v}  \sum_{j=m}^{M} (\rho^2-1) {1\over x^{2\lambda}\abs{a_j
   }} \\ &  \le  C_{v, \rho} \frac1{x^{2\lambda} a_m} \sum_{j=m}^{M}  \frac1{\rho^{j-m} }  \le C_{v, \rho} {1\over m_\lambda(I(x, a_m))},
\end{align*}
where we have used $\displaystyle \rho\le {a_{j+1}\over a_j}\le \rho^2.$

Now we shall prove (ii).  As in the proof of (i), we can get the desired estimation by considering the following two cases:

Case 1: $x\le 2a_{k+1}.$ In this case,
$$m_\lambda(I(x,a_k))\sim a_k^{2\lambda+1}.$$
By the mean value theorem and Lemma \ref{KerPoisson} $iii)$, there exists $a_j\le \xi_j\le a_{j+1}$ such that
\begin{align*}
& \abs{\sum_{j=-M}^{m-1}v_j \left(\P_{a_{j+1}}(x,y)- \P_{a_{j}}(x,y) \right) }\\&\le C \norm{v}_{\ell^\infty(\mathbb Z)}\sum_{j=-M}^{m-1}(a_{j+1}-a_j)\abs{\partial_t \P_t(x, y)\big|_{t=\xi_j}}\\&\le C_v(\rho^2-1)\sum_{j=-M}^{m-1} {a_j \over |x-y|^{{2\lambda+2}} }   \le C_{v, \rho} \sum_{j=-M}^{m-1} {\frac{ a_j }{a_k }}\cdot {1\over a_k^{2\lambda+1}}\\
&
\le C_{v, \rho}  {1\over m_\lambda(I(x, a_k))}\rho^{-(k-m+1)}.
\end{align*}

Case 2: $x\ge 2a_{k+1}\ge 2|x-y|.$ In this case,
$$m_\lambda(I(x,a_k))\sim x^{2\lambda}a_k\quad \text{and} \quad x\sim y.$$

By the mean value theorem and Lemma \ref{KerPoisson} $iii)$, there exists $a_j\le \xi_j\le a_{j+1}$ such that
\begin{align*}
& \abs{\sum_{j=-M}^{m-1}v_j \left(\P_{a_{j+1}}(x,y)- \P_{a_{j}}(x,y) \right) }\\
&\le C \norm{v}_{\ell^\infty(\mathbb Z)}\sum_{j=-M}^{m-1}(a_{j+1}-a_j)\abs{\partial_t \P_t(x, y)\big|_{t=\xi_j}}\\
&\le C_v(\rho^2-1)\sum_{j=-M}^{m-1} {a_j \over (xy)^{\lambda}(|x-y|^2+\xi_j^2 )}\le C_{v, \rho} \sum_{j=-M}^{m-1} {\frac{ a_j }{a_k }}\cdot {1\over x^{2\lambda}a_k}\\
&
\le C_{v, \rho}  {1\over m_\lambda(I(x, a_k))}\rho^{-(k-m+1)},
\end{align*}
where we have used that $k \ge m.$
\end{proof}

\begin{proof}[Proof of Theorem \ref{Thm:Maximalcontrol}]
 Observe that, for any $x_0\in \R_+$ and $N=(N_1, N_2),$
$$T_N f(x_0)=T_{(N_1, M)} f(x_0)-T_{(N_2+1, M)} f(x_0),$$
with $-M\le N_1<N_2\le  M.$
Then, it suffices to estimate $\abs{T_{m, M} f(x_0)}$ for $\abs{m}\le M$ with constants independent of $x_0,$ $m$ and $M.$ Denote $I_k=I(x_0, a_k)$ for each $k\in \mathbb Z$.
Let us split $f$ as
\begin{align*}
f &=f \chi_{I_m} +f \chi_{I_m^c} =:f_1+f_2.
\end{align*}
 Then,  we have
\begin{align*}
\abs{T_{(m,M)} f(x_0)}&\le \abs{T_{(m,M)} f_1(x_0)}+\abs{T_{(m,M)} f_2(x_0)}\\
&=: I+II.
\end{align*}
For $I$, by Lemma \ref{cotlar} (i) we have
\begin{align*}
I&=\abs{T_{(m,M)} f_1(x_0)}= \abs{\int_{\real_+} \sum_{j=m}^{M}v_j \left(\P_{a_{j+1}}(x_0,y)-\P_{a_{j}}(x_0,y) \right)  f_1(y)dm_\lambda(y)}
\\&  \le C_{v, \rho}  {1\over m_\lambda(I_m)} \int_{I_m}  \abs{f(y)}dy \le C_{v, \rho, \lambda}\M f(x_0).
\end{align*}
For part $II$,
\begin{align*}\label{equ:II}
II&=\abs{T_{(m,M)} f_2(x_0)}=\frac{1}{m_\lambda(I_{m-1})}\int_{I_{m-1}} \abs{T_{(m,M)} f_2(x_0)}dm_\lambda(z) \\
&\le \frac{1}{ m_\lambda(I_{m-1})}\int_{I_{m-1}}\abs{T_{(-M,M)} f(z)}dm_\lambda(z)\\
 &\quad+\frac{1}{ m_\lambda(I_{m-1})}\int_{I_{m-1}}\abs{T_{(-M,M)} f_1(z)}dm_\lambda(z)\\
&\quad +\frac{1}{ m_\lambda(I_{m-1})}\int_{I_{m-1}}\abs{T_{(m,M)} f_2(z)-T_{(m,M)} f_2(x_0)}dm_\lambda(z)\\
&\quad +\frac{1}{ m_\lambda(I_{m-1})}\int_{I_{m-1}}\abs{T_{(-M,m-1)} f_2(z)}dm_\lambda(z)\\
&=:A_1+A_2+A_3+A_4.
\end{align*}
(If $m=-M$, we understand that $A_4=0$.) It is clear that
\begin{equation*}
A_1\le  \M (T_{(-M,M)} f)(x_0).
\end{equation*}
For $A_2,$ by the uniform boundedness of $T_{N}$, we get
\begin{multline*}
A_2\le \left(\frac{1}{ m_\lambda(I_{m-1})}\int_{I_{m-1}}\abs{T_{(-M,M)} f_1(z)}^qdm_\lambda(z)\right)^{1/q}\le C\left(\frac{1}{ m_\lambda(I_{m-1})}\int_{\mathbb{R}_+}\abs{f_1(z)}^qdm_\lambda(z)\right)^{1/q}\\
=C\left(\frac{1}{ m_\lambda(I_{m-1})}\int_{I_{m}}\abs{f(z)}^qdm_\lambda(z)\right)^{1/q}\le C\left(\frac{1}{ m_\lambda(I_{m-1})}\int_{I_{m}}\abs{f(z)}^qdm_\lambda(z)\right)^{1/q}\le C\M_q f(x_0).
\end{multline*}
For the third term $A_3$, with $z\in I_{m-1}$, by the mean value theorem ({with} $\xi:=\theta x_0+(1-\theta)z$ for certain $\theta\in (0,1)$) and Proposition \ref{Thm:KernelEst} we have
\begin{align*}
&\abs{T_{(m,M)} f_2(z)-T_{(m,M)} f_2(x_0)}\\
&=\abs{\int_{I_{m}^c} K_{(m,M)}(z, y)f(y)dm_\lambda(y)-\int_{I_m^c}  K_{(m,M)}(x_0, y)f(y)dm_\lambda(y) }\\
&\le \int_{I_m^c} \abs{K_{(m,M)}(z, y)-K_{(m,M)}(x_0, y)}\abs{f(y)}dm_\lambda(y)\\
&=\sum_{j=1}^{+\infty}\int_{I_{2^jm} \setminus I_{2^{j-1}m}} \abs{K_{(m,M)}(z, y)-K_{(m,M)}(x_0, y)}\abs{f(y)}dm_\lambda(y)\\
&=\sum_{j=1}^{+\infty}\int_{I_{2^jm} \setminus I_{2^{j-1}m}}  \abs{\partial_\xi K_{(m,M)}(\xi, y)}|z-x_0|\abs{f(y)}dm_\lambda(y) \\
&\le C\sum_{j=1}^{+\infty}\int_{I_{2^jm} \setminus I_{2^{j-1}m}}  {\abs{z-x_0} \over  { m_\lambda(I(\xi, |\xi-y|))|\xi-y| }}\abs{f(y)}dm_\lambda(y)\\
&\le C\sum_{j=1}^{+\infty} 2^{-{j}+2}\frac{1}{m_\lambda(I_{2^{j}m})} \int_{I_{2^{j}m}}\abs{f(y)}dm_\lambda(y)\\
&\le C\M f (x_0),
\end{align*}
where $I_{2^jm}=I(x_0, 2^ja_m)$ for any $j\ge 1,$ and we have used that $|z-x_0|\le {a_m/ \rho}$ and $|y-\xi|\ge 2^{j-2}a_m$ when $y\in I_{2^jm} \setminus I_{2^{j-1}m}.$
Then,
\begin{align*}
A_3=\frac{1}{m_\lambda(I_{m-1})}\int_{I_{m-1}}\abs{ T_{(m,M)} f_2(z)- T_{(m,M)} f_2(x_0)}dm_\lambda(z) \le C\M f(x_0).
\end{align*}
For the latest one,  $A_4,$  we have
\begin{align*}
A_4&=\frac{1}{m_\lambda(I_{m-1})}\int_{I_{m-1}}\abs{ T_{(-M,m-1)} f_2(z)}dm_\lambda(z)\\
&\le \frac{1}{m_\lambda(I_{m-1})}\int_{I_{m-1}}\int_{I_m^c}\abs{ K_{(-M,m-1)}(z, y) f(y)}dm_\lambda(y)dm_\lambda(z).
\end{align*}
Then, we consider the inner integral appeared in the above inequalities first. Since $z\in I_{m-1}$,  $y\in I_m^c$ and the sequence $\{a_j\}_{j\in \mathbb Z}$ is $\rho$-lacunary sequence, we have  $\abs{z-y}\sim \abs{y-x_0}.$
From this and by Lemma \ref{cotlar} (ii), we get
\begin{align*}
&\int_{I_m^c}\abs{K_{(-M,m-1)}(z, y) f(y)}dm_\lambda(y)\\
   &=\sum_{k=m}^{+\infty}\int_{I_{k+1}\setminus I_{k}} \abs{\sum_{j=-M}^{m-1}v_j \left( \P_{a_{j+1}}(z, y)-\P_{a_{j}}(z, y)\right) f(y)}dm_\lambda(y)\\
&\le C_{\rho, v} \sum_{k=m}^{+\infty}\rho^{-(k-m+1)}\left(\frac1{m_\lambda(I(x,a_k))}\int_{I_{k+1}\setminus I_{k}} \abs{f(y)}dm_\lambda(y)\right)\\
 &\le C_{\rho, v, \lambda} \M f(x_0)\sum_{k=m}^{+\infty}\rho^{-(k-m+1)}\\
 &\le C_{\rho, v, \lambda}\M f(x_0).
\end{align*}
Hence,
$$A_4\le C\M f(x_0).$$
Combining the estimates above for $A_1, A_2, A_3$ and $A_4$, we get
$$II\le \M ( T_{(-M,M)} f)(x_0)+C \M_q f(x_0).$$
And then we have
$$ \abs{ T_{(m,M)} f(x_0)}\le C\left( \M (T_{(-M,M)} f)(x_0)+\M_q f(x_0)\right). $$
As the constants $C$ appeared above all only depend on $\norm{v}_{l^\infty(\mathbb Z)}$, $\rho$ and $\lambda$, we have  proved that
$$ T^{*}_M f(x_0)\le C\left\{\M( T_{(-M, M)} f)(x_0)+\M_q f(x_0)\right\}.$$ This complete the proof of the theorem.
\end{proof}

Now, we are in a position to prove Theorem \ref{Thm:PoissonLp}.
\begin{proof}[Proof of Theorem \ref{Thm:PoissonLp}] {Given $\omega \in A_p$, we choose $1<q<p $ such that $\omega \in A_{p/q}$.
Then it is well known that, the maximal operators $\M$ and $ \M_q$ are bounded on $L^p(\real_+, \omega dm_\lambda),$ see \cite{Duo}.}  On the other hand, since the operators $T_N$ are uniformly bounded in $L^p(\real_+, \omega dm_\lambda)$ with $\omega \in A_p$, we have
\begin{align*}
\norm{T_M^*f}_{L^p(\real_+, \omega dm_\lambda)}&\le C\left(\norm{\M (T_{(-M, M)} f)}_{L^p(\real_+, \omega dm_\lambda)}+\norm{\M_q f}_{L^p(\real_+, \omega dm_\lambda)}\right)\\&\le C\left(\norm{T_{(-M, M)} f}_{L^p(\real_+, \omega dm_\lambda)}+\norm{f}_{L^p(\real_+, \omega dm_\lambda)}\right)\le C\norm{f}_{L^p(\real_+, \omega dm_\lambda)}.
\end{align*}
Note that the constants $C$ appeared above do not depend on $M$. Consequently, letting $M$ increase to infinity, we get the proof of the $L^p$ boundedness of the maximal operator $T^*.$
This completes the proof of part $(a)$ of the theorem.

In order to prove $(b)$, we consider the $\ell^\infty(\mathbb Z^2)$-valued operator
$\mathcal{T}f(x) = \{  T_N f(x) \}_{N\in \mathbb Z^2}$. Since  $\|\mathcal{T}f(x) \|_{\ell^\infty(\mathbb Z^2)}= T^*f(x)$,   by using $(a)$ we know that the operator $\mathcal{T}$ is bounded from $L^p(\real_+, \omega dm_\lambda) $ into $L^p_{\ell^\infty(\mathbb Z^2)}(\mathbb{R}_+, \omega dm_\lambda) $, for every $1<p<\infty$ and $\omega \in A_p$. The kernel of the operator $\mathcal{T}$ is given by $\mathcal{K}(x) = \{ K_N(x)\} _{N\in \mathbb Z^2}$. Therefore, by the vector valued Calder\'on-Zygmund theory, the operator $\mathcal{T}$ is bounded from $L^1(\mathbb{R}_+, \omega dm_\lambda)$ into weak-$L^1_{\ell^\infty(\mathbb Z^2)}(\mathbb{R}_+, \omega dm_\lambda)$ for $\omega \in A_1$. Hence, as $\|\mathcal{T}f(x) \|_{\ell^\infty(\mathbb Z^2)}= T^*f(x)$, we get the proof of  $(b)$.

For $(c)$, we shall prove that, if $f\in BMO(\mathbb R_+, dm_\lambda)$ and there exists $x_0\in \mathbb R_+$ such that $T^*f(x_0)<\infty,$ then $T^*f(x)<\infty$ for $a.e.$ $x\in \real_+.$
Set $I=I(x_0, 4\abs{x-x_0})$ with $x\neq x_0$. And we decompose $f$ to be
\begin{equation*}
f=(f-f_I)\chi_I+(f-f_I)\chi_{I^c}+f_I=:f_1+f_2+f_3.
\end{equation*}
Note that $T^*$ is $L^p$ bounded for any $1<p<\infty.$ Then $T^*f_1(x)<\infty$, because $f_1\in L^p(\mathbb R_+, dm_\lambda)$, for any $1<p<\infty.$ And  $T^* f_3=0$, since $\P_{a_j}f_3=f_3$ for any $j\in \mathbb Z.$
On the other hand by the smoothness properties of the kernel , we have
\begin{align*}
&\Big|T_N f_2(x)-T_N f_2(x_0)\Big|=\Big| \int_{I^c}\left(K_N(x, y) - K_N (x_0, y)\right)f_2(y)dm_\lambda(y)\Big|\\
&\le C \int_{I^c}  \frac{\abs{x-x_0}}{m_\lambda(I(x_0, |y-x_0|))\abs{y-x_0}}\abs{f(y)-f_I} dm_\lambda(y)\\
&\le C \sum_{k=1}^{+\infty}{ |x-x_0|} \int_{2^{k} I\setminus 2^{k-1}I}  {\abs{f(y)-f_I}\over {m_\lambda(I(x_0, |y-x_0|))\abs{y-x_0}}}dm_\lambda(y)\\
&\le C \sum_{k=1}^{+\infty}{ |x-x_0|\over m_\lambda(I(x_0, 2^{k+1}|x-x_0|))2^{k+1}|x-x_0|} \int_{2^kI}  {\abs{f(y)-f_I}} dm_\lambda(y) \\ &\le C \sum_{k=1}^{+\infty}2^{-(k+1)}{ 1\over m_\lambda(I(x_0, 2^{k+1}|x-x_0|))} \int_{2^{k}I} \Big(  {\abs{f(y)-f_{2^kI}}}+\sum_{l=1}^{k}\abs{f_{2^lI}-f_{2^{l-1}I}}\Big)dm_\lambda(y)\\
&\le C\sum_{k=1}^{+\infty}2^{-(k+1)}{ 1\over m_\lambda(I(x_0, 2^{k+1}|x-x_0|))} \int_{2^{k}I} \Big(  {\abs{f(y)-f_{2^kI}}}+2k\norm{f}_{BMO(\real_+, dm_\lambda)}\Big)dm_\lambda(y)\\
&\le C\sum_{k=1}^{+\infty}2^{-(k+1)}{(1+2k)\norm{f}_{BMO(\real_+, dm_\lambda)}}\\
&\le  C\norm{f}_{BMO(\real_+,dm_\lambda)},
 \end{align*}
where $2^k I=I(x_0, 2^{k}\cdot 4|x-x_0|)$ for any $k\in \mathbb N.$
 Hence
  \begin{align*}
\norm{T_N f_2(x)-T_N f_2(x_0)}_{l^\infty(\mathbb Z^2)} \le C\norm{f}_{BMO(\mathbb R_+, dm_\lambda)},
\end{align*}
and therefore $ T^*f(x) = \norm{T_N f(x)}_{l^\infty(\mathbb Z^2)}  \le C < \infty.$

Finally,  the  estimate (\ref{sharp}) can be  proved in a parallel way to the proof of part $(d)$ of Theorem \ref{Thm:BMO}, since $\mathcal{T}1(x) = \{T_N1(x)\}= 0$(also see \cite{MTX}).
\end{proof}

From Theorem \ref{Thm:PoissonLp}, we can get the following consequence.

\begin{thm}\label{Thm:ae}\begin{enumerate}[(a)]
    \item If $1<p<\infty$ and $\omega\in A_p$, then $T_N f$ converges {\it{a.e.}} and  in $L^p(\mathbb R_+, \omega dm_\lambda)$ norms for all $f\in L^p(\mathbb R_+, \omega dm_\lambda)$ as $N=(N_1,N_2)$ tends to $(-\infty, +\infty).$
   \item If $p=1$  and $\omega\in A_1$, then $T_N f$ converges {\it{a.e.}} and in measure for all $f\in L^1(\mathbb R_+, \omega dm_\lambda)$ as $N=(N_1,N_2)$ tends to $(-\infty, +\infty).$
\end{enumerate}
\end{thm}

\begin{proof}
First, we shall see that if $\varphi$ is a test function, then $T_N \varphi(x)$ converges for all $x\in \real_+$.  In order to prove this, it is enough to see that for any  $(L,M)$ with $0<L<M$,  the  series \begin{equation*} A= \sum_{j=L}^M v_j ( \P_{a_{j+1}} \varphi(x) - \P_{a_j} \varphi(x))\quad   \hbox{  and  }\quad  B= \sum_{j=-M}^{-L} v_j ( \P_{a_{j+1}} \varphi(x) - \P_{a_j} \varphi(x))
 \end{equation*}
converge to zero, when $L, M\rightarrow +\infty$.
By the mean value theorem and the $\rho$-lacunarity  of the sequence $\{a_j\}_{j\in \mathbb Z}$, we have
\begin{align*}
|A|   & \le   C_{ v}  { \int_{\real_+}\sum_{j=L}^{M}   \abs{\P_{a_{j+1}}(x,y)-\P_{a_{j}}(x,y)} |\varphi(y)|dm_\lambda(y)}\\
&\le C_{ v, \rho, \lambda}   \int_{\real_+} \sum_{j=L}^{M}  \frac{C}{a^{2\lambda+1}_j} |\varphi(y)|dm_\lambda(y) \le  C_{ v, \rho, \lambda}\left({1\over a_L^{2\lambda+1}}\sum_{j=L}^{M} \frac{a_L^{2\lambda+1}}{a_j^{2\lambda+1}}\right) \int_{\real_+} |\varphi(y)|dm_\lambda(y)\\
&\le C_{ v, \rho, \lambda} {\rho^{{2\lambda+1}}\over {\rho^{{2\lambda+1}}-1}}\|\varphi\|_{L^1(\real_+, dm_\lambda)}{1\over a_L^{{2\lambda+1}}} \longrightarrow 0, \quad \hbox{as}\ { L,M \to +\infty}.
\end{align*}
On the other hand,  as the integral of the kernels are zero, we can write
\begin{align*}
B&= \int_{\real_+} \sum_{j=-M}^{-L}v_j \left(\P_{a_{j+1}}(x,y)- \P_{a_{j}}(x,y) \right)( \varphi(y)- \varphi(x))dm_\lambda(y)\\
&= \sum_{j=-M}^{-L}\Big(\int_{0}^{\sqrt a_j}v_j \left(\P_{a_{j+1}}(x,y)- \P_{a_{j}}(x,y) \right)( \varphi(y)- \varphi(x))dm_\lambda(y)\\
&\quad \quad +\int^{+\infty}_{\sqrt a_j}v_j \left(\P_{a_{j+1}}(x,y)- \P_{a_{j}}(x,y) \right)( \varphi(y)- \varphi(x))dm_\lambda(y)\Big)\\
&=:B_1+B_2.
\end{align*}
And, then proceeding as in the case $A$, and by using the fact that $\varphi$ is a test function,  we have
\begin{align*}
|B_1|&\le C_{\lambda, \rho} \norm{v}_{\ell^\infty(\mathbb Z)} \sum_{j=-M}^{-L} \int_{0}^{\sqrt a_j} { a_j\over (|x-y|^2+a_j^2)^{\lambda+1} }\abs{ \varphi(y)- \varphi(x)}dm_\lambda(y)\\
&\le C_{\lambda, \rho} \norm{\nabla\varphi}_{L^\infty(\real_+, dm_\lambda)}\sum_{j=-M}^{-L} \int_{0}^{\sqrt a_j} { a_j|x-y|\over (|x-y|^2+a_j^2)^{\lambda+1} }dm_\lambda(y)\\
&\le C_{ \lambda, \rho, \varphi} a_{-L}^{1/2}\sum_{j=-M}^{-L} {a_j^{1/2}\over  a_{-L}^{1/2}}\\
&\le C_{ \lambda, \rho, \varphi} {\sqrt \rho\over \sqrt \rho-1} a_{-L}^{1/2}\longrightarrow 0, \quad \hbox{as}\ { L,M \to +\infty}.
\end{align*}
On the other hand,
\begin{align*}
|B_2|&\le C_{\lambda, \rho} \norm{v}_{\ell^\infty(\mathbb Z)} \sum_{j=-M}^{-L} \int_{\sqrt a_j}^{+\infty} { a_j\over (|x-y|^2+a_j^2)^{\lambda+1} }\abs{ \varphi(y)- \varphi(x)}dm_\lambda(y)\\
&\le C_{\lambda, \rho} \norm{\varphi}_{L^\infty(\real_+, dm_\lambda)}\sum_{j=-M}^{-L} \int_{\sqrt a_j}^{+\infty} { a_j \over (|x-y|^2+a_j^2)^{\lambda+1} }dm_\lambda(y)\\
&\le C_{ \lambda, \rho, \varphi} a_{-L}^{1/2}\sum_{j=-M}^{-L} {a_j^{1/2}\over  a_{-L}^{1/2}}\\
&\le C_{ \lambda, \rho, \varphi} {\sqrt \rho\over \sqrt \rho-1} a_{-L}^{1/2}\longrightarrow 0, \quad \hbox{as}\ { L,M \to +\infty}.
\end{align*}

As the set of test functions is dense in $L^p(\mathbb{R}_+, dm_\lambda)$, by Theorem \ref{Thm:PoissonLp} we get the $a.e.$ convergence for any function in $L^p(\mathbb{R}_+, dm_\lambda)$. Analogously, since $L^p(\mathbb{R}_+) \cap L^p(\mathbb{R}_+, dm_\lambda)) $ is dense in $L^p(\mathbb{R}_+, \omega dm_\lambda)$, we get the $a.e.$ convergence for functions in $L^p(\mathbb{R}_+, \omega dm_\lambda) $ with $1\le p<\infty$. By using the dominated convergence theorem,  we can prove the convergence in $L^p(\mathbb{R}_+, \omega dm_\lambda)$ norm for $1<p<\infty$, and also in measure.
\end{proof}

\section{Local growth of the Maximal operator $T^*$}\label{Sec:local-growth}

In this section, we give the proof of Theorem \ref{Thm:GrothLinfinity}.

\begin{proof}[Proof of Theorem \ref{Thm:GrothLinfinity}]
We will prove it only in the case $1<p<\infty$. For the cases $p=1$ and  $p=\infty$, the proof follows by introducing the obvious changes.  Since $2r<1,$ we know that $\tilde I\backslash \tilde I_{2r}\neq \emptyset.$ Let $f=f_1+f_2$, where $f_1=f\chi_{\tilde I_{2r}}$ and $f_2=f(x)\chi_{\tilde I\backslash \tilde I_{2r}}$. Then
$$\abs{T^{*} f(x)}\le \abs{T^{*} f_1(x)}+\abs{T^{*} f_2(x)}.$$
By Theorem \ref{Thm:PoissonLp}, we have
\begin{multline*}
\frac{1}{m_\lambda(\tilde I_r)} \int_{\tilde I_r} \abs{T^{*} f_1 (x)} dm_\lambda(x)\le \left(\frac{1}{m_\lambda(\tilde I_r)} \int_{\tilde I_r} \abs{T^{*} f_1 (x)}^2 dm_\lambda(x)\right)^{1/2}
\\\le C \left(\frac{1}{m_\lambda(\tilde I_r)} \int_{\mathbb{R}_+}\abs{ f_1 (x)}^2 dm_\lambda(x)\right)^{1/2}\le C\norm{f}_{L^\infty(\mathbb{R}_+, dm_\lambda)}.
\end{multline*}

We should note that, for any $j\in \mathbb Z,$ by Lemma \ref{KerPoisson},
 \begin{multline}\label{equ:constant}
 \int_{\real_+}\abs{\P_{a_{j+1}}(x,y) -\P_{a_{j}}(x,y) }~ dm_\lambda(y)=\int_{\real_+}\abs{\int_{a_j}^{a_{j+1}}\partial_t\P_{t}(x,y) dt }~ dm_\lambda(y) \\
 \le C_\lambda\int_{\real_+} {a_{j+1}-a_j\over {(|x-y|^2+a_j^2)^{\lambda+1}}} dm_\lambda(y)  =C_{\lambda, \rho}.
 \end{multline}
And for $1< p < \infty$ and any $N=(N_1, N_2)$, by H\"older's inequality, \eqref{equ:constant}, Fubini's Theorem and Lemma  \ref{KerPoisson} we have
\begin{align*}
&\abs{\sum_{j=N_1}^{N_2}v_j\left(\P_{a_{j+1}}f_2(x)- \P_{a_j}f_2(x)\right)}\nonumber\\
&\le C\sum_{j=N_1}^{N_2} \abs{v_j \int_{\real_+}\left(\P_{a_{j+1}}(x,y) -\P_{a_{j}}(x,y)\right)  f_2(y)~ dm_\lambda(y)}\nonumber\\
&\le C\norm{v}_{l^p(\mathbb Z)}\left( \sum_{j=N_1}^{N_2}\left(\int_{\real_+}\abs{\P_{a_{j+1}}(x,y) -\P_{a_{j}}(x,y)} \abs{f_2(y)}~ dm_\lambda(y)\right)^{p'}\right)^{1/p'} \nonumber \\
&\le C_v\Big(\sum_{j=N_1}^{N_2} \Big\{\int_{\real_+}\abs{\P_{a_{j+1}}(x,y) -\P_{a_{j}}(x,y) } \abs{f_2(y)}^{p'}~ dm_\lambda(y)\Big\}  \\
 & \quad \quad  \times  \Big\{\int_{\real_+}\abs{\P_{a_{j+1}}(x,y) -\P_{a_{j}}(x,y) }~ dm_\lambda(y)  \Big\}^{p'/p} \Big)^{1/p'}
\nonumber \\
&\le C_{v,p, \lambda, \rho}\left(\sum_{j=N_1}^{N_2} \int_{\real_+}\abs{\P_{a_{j+1}}(x,y) -\P_{a_{j}}(x,y) } \abs{f_2(y)}^{p'}~ dm_\lambda(y)   \right)^{1/p'}
\nonumber \\
&\le C_{v,p, \lambda, \rho} \left(\int_{\real_+}\sum_{j=-\infty}^{+\infty} \abs{\P_{a_{j+1}}(x,y) -\P_{a_{j}}(x,y)} \abs{f_2(y)}^{p'}~ dm_\lambda(y)  \right)^{1/p'}
\nonumber \\ \nonumber
&\le C_{v,p, \lambda, \rho}\left(\int_{\real_+} \frac1{|x-y|^{2\lambda+1}} \abs{f_2(y)}^{p'}~ dm_\lambda(y)  \right)^{1/p'}.
\end{align*}
For $y\in \tilde I\backslash \tilde I_{2r}$ and $x\in \tilde I_r$,  we have  $r\le |x-y|\le 2$. Then,  we get
\begin{align*}
&\frac{1}{m_\lambda(\tilde I_r)} \int_{\tilde I_r} \abs{T^{*} f_2(x)}dm_\lambda(x) \\&\le
C_{v,p, \lambda, \rho}\frac{1}{m_\lambda(\tilde I_r)} \int_{\tilde I_r} \left(\int_{\real_+} \frac1{|x-y|^n} \abs{f_2(y)}^{p'}~ dm_\lambda(y)  \right)^{1/p'}dm_\lambda(x) \\
&\le C_{v,p, \lambda, \rho}\frac{\norm{f}_{L^\infty(\real_+, dm_\lambda)}}{m_\lambda(\tilde I_r)} \int_{\tilde I_r} \left(\int_{r\le |x-y| \le 2} \frac1{|x-y|^{2\lambda+1}} ~ dm_\lambda(y) \right)^{1/p'}dm_\lambda(x) \\
 &\sim \Big(\log\frac{2}{r}\Big)^{1/p'}\norm{f}_{L^\infty(\real_+, dm_\lambda)}.
\end{align*}
Hence, $$\frac{1}{m_\lambda(\tilde I_r)} \int_{\tilde I_r} \abs{T^{*} f(x)}dm_\lambda\le C\left(1+\Big(\log\frac{2}{r}\Big)^{1/p'}\right)\norm{f}_{L^\infty(\real_+, dm_\lambda)}  \le C\Big(\log\frac{2}{r}\Big)^{1/p'}\norm{f}_{L^\infty(\real_+, dm_\lambda)}.$$
\end{proof}

\vspace{1em}

\vspace{3em}

\end{document}